\documentclass[leqno,11pt]{article}

\usepackage[utf8]{inputenc}
\usepackage[T1]{fontenc}
\usepackage{microtype}

\usepackage[a4paper]{geometry}

\ifdefined\screenview
  \edef\mtht{\the\textheight}
  \edef\mtwd{\the\textwidth}
\fi

\usepackage[dvipsnames]{xcolor}

\usepackage{amsmath}
\usepackage{amsthm}
\usepackage{tikz-cd}
\usetikzlibrary{arrows} 
\tikzset{
  commutative diagrams/.cd, 
  arrow style=tikz, 
  diagrams={>=stealth}
}
\usetikzlibrary{matrix,decorations.pathreplacing,calc}

\usepackage{textcomp}
\usepackage[sb]{libertine}
\usepackage[varqu,varl]{zi4}%
\usepackage[libertine,bigdelims,vvarbb]{newtxmath}
\usepackage[supstfm=libertinesups,supscaled=1.2,raised=-.13em]{superiors}
\useosf

\usepackage[scr=boondox,cal=euler]{mathalfa}

\usepackage{marvosym}

\usepackage{slashed}
\usepackage{esint} 
\usepackage[english]{babel}

\usepackage{imakeidx}
\indexsetup{noclearpage}
\makeindex[intoc]

\usepackage{csquotes}
\usepackage[
  backend=biber,
  hyperref=true,
  backref=true,
  isbn=false,
  doi=true,
  natbib=true,
  eprint=true,
  useprefix=true,
  maxcitenames=99,
  maxbibnames=99,  
  maxalphanames=99, 
  minalphanames=99,
  safeinputenc,
  style=alphabetic,
  citestyle=alphabetic,
  block=space,
  datamodel=preamble/ext-eprint
]{biblatex}
\usepackage[
  bookmarksnumbered = true,
  hypertexnames = false,
  colorlinks    = true,
  citecolor     = gray,
  linkcolor     = gray,
  urlcolor      = gray,
  breaklinks
]{hyperref}

\DeclareFieldFormat{url}{%
  \href{#1}{\ComputerMouse}
}
\DeclareFieldFormat{doi}{%
  \mkbibacro{DOI}\addcolon\space
  \href{https://doi.org/#1}{#1}
}
\makeatletter
\DeclareFieldFormat{arxiv}{%
  arXiv\addcolon\space
  \href{http://arxiv.org/\abx@arxivpath/#1}{#1}
}
\makeatother
\DeclareFieldFormat{mr}{%
  MR\addcolon\space
  \href{http://www.ams.org/mathscinet-getitem?mr=MR#1}{#1}
}
\DeclareFieldFormat{zbl}{%
  Zbl\addcolon\space
  \href{http://zbmath.org/?q=an:#1}{#1}
}
\renewbibmacro*{eprint}{%
  \printfield{arxiv}%
  \newunit\newblock
  \printfield{mr}%
  \newunit\newblock
  \printfield{zbl}%
  \newunit\newblock
  \iffieldundef{eprinttype}
  {\printfield{eprint}}
  {\printfield[eprint:\strfield{eprinttype}]{eprint}}
}

\AtEveryBibitem{%
  \clearlist{publisher}%
}
\AtEveryBibitem{%
  \clearlist{address}%
}
\DeclareFieldFormat[article,inproceedings,inbook,incollection,thesis]{title}{\textit{#1}}
\renewbibmacro{in:}{}
\newcommand{\printreferences}{\printbibliography[heading=bibintoc]}

\usepackage[inline,shortlabels]{enumitem}

\usepackage{subcaption}

\usepackage[yyyymmdd]{datetime}

\usepackage{etoolbox}
\ifundef{\abstract}{}{\patchcmd{\abstract}%
    {\quotation}{\quotation\noindent\ignorespaces}{}{}}

\usepackage[super]{nth}

\usepackage{thmtools}

\numberwithin{equation}{section}

\renewcommand{\qedsymbol}{$\blacksquare$}

\newcommand{\CorollaryQED}{\qedsymbol}
\newcommand{\ConjectureQED}{$\square$}
\newcommand{\SituationQED}{$\times$}
\newcommand{\DefinitionQED}{$\bullet$}
\newcommand{\NotationQED}{$\circ$}
\newcommand{\ExampleQED}{$\spadesuit$}
\newcommand{\RemarkQED}{$\clubsuit$}

\declaretheorem[numberlike=equation,]{theorem}
\declaretheorem[numbered=no,name=Theorem]{theorem*}
\declaretheorem[numberlike=equation,name=Lemma]{lemma}
\declaretheorem[numberlike=equation,name=Proposition]{prop}
\declaretheorem[numberlike=equation,name=Corollary,qed=\CorollaryQED]{cor}

\declaretheorem[numberlike=equation,name=Definition,style=definition,qed=\DefinitionQED]{definition}
\declaretheorem[numbered=no,name=Definition,style=definition,qed=\DefinitionQED]{definition*}
\declaretheorem[numberlike=equation,name=Notation,style=definition,qed=\NotationQED]{notation}

\declaretheorem[numberlike=equation,style=definition,qed=\ExampleQED]{example}

\declaretheorem[numberlike=equation,style=remark,qed=\RemarkQED]{remark}

\declaretheorem[numberlike=equation,style=remark]{convention}

\def\makeautorefname#1#2{\AtBeginDocument{\expandafter\def\csname#1autorefname\endcsname{#2}}}
\makeautorefname{table}{Table}        
\makeautorefname{chapter}{Chapter}
\makeautorefname{section}{Section}
\makeautorefname{subsection}{Section}
\makeautorefname{subsubsection}{Section}
\makeautorefname{footnote}{Footnote}
\AtBeginDocument{\def\itemautorefname~#1\null{(#1)\null}}
\AtBeginDocument{\def\equationautorefname~#1\null{(#1)\null}}

\numberwithin{substep}{step}
\makeautorefname{step}{Step}
\makeautorefname{substep}{Step}

\makeautorefname{case}{Case}
\makeautorefname{substep}{Step}

\setlist[description]{leftmargin=!,labelindent=1em}
\setlist[enumerate]{label={\rm (\arabic*)},ref=\arabic*}
\setlist[enumerate,2]{label={\rm (\alph*)},ref=\theenumi.\alph*}
\let\C\undefined

\usepackage{bm}
\usepackage{mathtools} 
\usepackage{stmaryrd} 

\DeclareFontFamily{U}{mathx}{\hyphenchar\font45}
\DeclareFontShape{U}{mathx}{m}{n}{
      <5> <6> <7> <8> <9> <10>
      <10.95> <12> <14.4> <17.28> <20.74> <24.88>
      mathx10
      }{}
\DeclareSymbolFont{mathx}{U}{mathx}{m}{n}
\DeclareFontSubstitution{U}{mathx}{m}{n}
\DeclareMathAccent{\widecheck}{0}{mathx}{"71}
\DeclareMathAccent{\wideparen}{0}{mathx}{"75}

\DeclareMathOperator{\Diff}{Diff}

\DeclareMathOperator{\HF}{\HF}

\DeclareMathOperator{\Hom}{Hom}

\DeclareMathOperator{\area}{area}

\DeclareMathOperator{\ind}{index}

\DeclareMathOperator{\supp}{supp}

\DeclarePairedDelimiter\paren{\lparen}{\rparen}

\DeclarePairedDelimiter{\Abs}{\|}{\|}

\DeclarePairedDelimiter{\Inner}{\langle}{\rangle}
\DeclarePairedDelimiter{\abs}{\lvert}{\rvert}
\DeclarePairedDelimiter{\bracket}{\langle}{\rangle}

\DeclarePairedDelimiter{\set}{\lbrace}{\rbrace}
\def\({\left(}
\def\){\right)}
\def\<{\left\langle}
\def\>{\right\rangle}

\newcommand{\C}{{\mathbf{C}}}

\newcommand{\N}{{\mathbf{N}}}

\newcommand{\R}{\mathbf{R}}

\newcommand{\co}{\mskip0.5mu\colon\thinspace}

\newcommand{\defined}[2][\key]{\def\key{#2}\textbf{#2}\index{#1}}
\newcommand{\delbar}{\bar{\del}}

\newcommand{\del}{\partial}

\newcommand{\id}{\mathrm{id}}

\newcommand{\inner}[2]{\bracket{#1, #2}}

\newcommand{\iso}{\cong}

\newcommand{\loc}{\mathrm{loc}}

\newcommand{\qandq}{\quad\text{and}\quad}

\newcommand{\qforq}{\quad\text{for}\quad}

\newcommand{\qforeveryq}{\quad\text{for every}\quad}

\newcommand{\reg}{\mathrm{reg}}

\newcommand{\sing}{\mathrm{sing}}

\newcommand{\vol}{\mathrm{vol}}

\renewcommand{\emptyset}{\varnothing}
\renewcommand{\epsilon}{\varepsilon}
\renewcommand{\setminus}{{\backslash}}

\renewcommand{\leq}{\leqslant}
\renewcommand{\geq}{\geqslant}

\makeatletter
\renewcommand*\env@matrix[1][*\c@MaxMatrixCols c]{%
  \hskip -\arraycolsep
  \let\@ifnextchar\new@ifnextchar
  \array{#1}}

\renewcommand\xleftrightarrow[2][]{%
  \ext@arrow 9999{\longleftrightarrowfill@}{#1}{#2}}
\newcommand\longleftrightarrowfill@{%
  \arrowfill@\leftarrow\relbar\rightarrow}
\makeatother

\newcommand{\rd}{{\rm d}}

\newcommand{\bM}{{\mathbf{M}}}

\newcommand{\sH}{\mathscr{H}}

\newcommand{\sJ}{\mathscr{J}}

\newcommand{\sR}{\mathscr{R}}

\newcommand{\fd}{{\mathfrak d}}

\newcommand{\fn}{{\mathfrak n}}

\author{
  Aleksander Doan
  \and
  Thomas Walpuski
}
\title{  
  Castelnuovo's bound and rigidity in almost complex geometry
}

\date{2020-09-21}

\begin{document}
\maketitle

\begin{abstract}
  This article is concerned with the question of whether an energy bound implies a genus bound for pseudo-holomorphic curves in almost complex manifolds.
  After reviewing what is known in dimensions other than six,
  we establish a new result in this direction in dimension six;
  in particular, for symplectic Calabi--Yau $3$--folds.
  The proof relies on compactness and regularity theorems for pseudo-holomorphic currents. 
\end{abstract}


\section{Introduction}
\label{Sec_Introduction}

In 1889, \citet{Castelnuovo1889} found a sharp upper bound for the genus of an irreducible, nondegenerate curve of a given degree in $\C P^n$;
see \cite[Chapter III Section 2]{Arbarello1985} for a proof in modern language.
A corollary of this result is that for every projective variety there is an upper bound for the genus of an irreducible curve representing a given homology class. 
Our starting point is the question:
\begin{center}
  Are there analogues of Castelnuovo's bound in almost complex geometry?
\end{center}

For curves in $\C P^2$ Castelnuovo's bound reduces to the degree-genus formula.
The latter is a consequence of the adjunction formula,
which generalizes to an inequality for almost complex $4$--manifolds \cite[Theorem 2.6.4]{McDuff2012}.
The adjunction inequality directly implies the following well-known genus bound.
\begin{prop}
  \label{Prop_FourDimensionalCastelnuovo}
  Suppose that $(M,J)$ is an almost complex $4$--manifold.
  If there exists a simple $J$--holomorphic map $u\co \Sigma \to M$ representing $A \in H_2(M)$,
  then the genus $g(\Sigma)$ satisfies
  \begin{equation}
    \label{Eq_FourDimensionalCastelnuovo}
    g(\Sigma) \leq \frac{1}{2}\(A\cdot A - \inner{c_1(M,J)}{A}\) + 1.
  \end{equation}
\end{prop}

The following is a consequence of Gromov's h-principle for symplectic embeddings \cite[Section 3.4.2 Theorem (A)]{Gromov1986}.
It shows that in higher dimensions there cannot be a genus bound which holds for all almost complex structures.

\begin{prop}
  \label{Prop_NoUniversalGenusBound}
  Let $(M,\omega)$ be a symplectic manifold of dimension $2n \geq 6$.
  For every $A \in H_2(M)$ with $\inner{[\omega]}{A} > 0$ and every $g \in \N$ there is an almost complex structure $J$ compatible with $\omega$ and a $J$--holomorphic embedding $u \co \Sigma \to M$ satisfying
  \begin{equation*}
    g(\Sigma) \geq g.
  \end{equation*}
\end{prop}

There are, however, genus bounds for \emph{generic} almost complex structures.
Here is a simple example, which follows easily from the index formula and transversality theorem for simple $J$--holomorphic maps  \cite[Chapter 3]{McDuff2012}.

\begin{prop}
  \label{Prop_HighDimensionalCastelnuovo}
  Let $M$ be a manifold of dimension $2n$.
  Denote by $\sJ$ the space of smooth almost complex structures on $M$ equipped with the $C_\loc^\infty$--topology.
  There is a comeager%
  \footnote{%
    Let $X$ be a topological space.
    A subset $A \subset X$ is called \defined{comeager} (or \defined{residual}) if it contains the intersection of countably many dense open subsets.
    A comeager subset of a complete metric space is dense.
  }%
  subset $\sJ_\clubsuit \subset \sJ$ such that for every $J \in \sJ_\clubsuit$ the following holds:
  if there exists a simple $J$--holomorphic map $u\co \Sigma \to M$ representing $A \in H_2(M)$,
  then
  \begin{equation}
    \label{Eq_HighDimensionalCastelnuovo}
    \begin{dcases}
      \inner{c_1(M,J)}{A} \geq 0 & \textnormal{if } n = 3 \\
      g(\Sigma) \leq \frac{\inner{c_1(M,J)}{A}}{n-3} + 1 & \textnormal{if } n > 3.
    \end{dcases}
  \end{equation}
  Moreover, if $M$ carries a symplectic form $\omega$,
  then the same holds with $\sJ$ replaced by the space $\sJ(\omega)$ of smooth almost complex structures compatible with $\omega$.
\end{prop}

\autoref{Prop_NoUniversalGenusBound} and \autoref{Prop_HighDimensionalCastelnuovo} are both well-known and we omit their proofs.

The preceding discussion leaves open the case of generic almost complex structures in dimension six and homology classes satisfying $\inner{c_1(M,J)}{A} \geq 0$.
In the present article,
we focus on the case
\begin{equation*}
  \inner{c_1(M,J)}{A} = 0,
\end{equation*}
that is: on classes for which the corresponding moduli space of $J$--holomorphic maps has expected dimension zero.
This includes all homology classes in symplectic Calabi--Yau $3$--folds, that is: symplectic manifolds $(M,\omega)$ such that $\dim M = 6$ and $c_1(M,J) = 0$ for some almost complex structure $J$ compatible with $\omega$. 
Our motivation for considering this case comes from our project to construct a symplectic analogue of the Pandharipande--Thomas invariants of projective Calabi--Yau $3$--folds \cite[Section 7]{Doan2017d}.
Another motivation comes from the Gopakumar--Vafa conjecture.
\citet{Bryan2001} defined the Gopakumar--Vafa BPS invariants $n_A^g(M,\omega)$ of a symplectic Calabi--Yau $3$--fold $(M,\omega)$ in terms of its Gromov--Witten partition function.
They conjectured that the BPS invariants $n_A^g(M,\omega)$ are integers and vanish for all but finitely many $g$ \cite[Conjecture 1.2]{Bryan2001}.
The integrality conjecture has been proved by \citet{Ionel2018}.
The finiteness conjecture remains open and is closely related to the question about the existence of genus bounds for symplectic Calabi--Yau $3$--folds.

Motivated by Gromov--Witten theory,
\citet{Bryan2001} introduced the notion of \emph{$k$--rigidity} for almost complex structures;
see \autoref{Def_KRigidAlmostComplexStructure}.
They conjectured that a generic almost complex structure is \emph{$\infty$--rigid} (or \emph{super-rigid}), that is: $k$--rigid for every $k \in \N$.
This has recently been proved by \citet{Wendl2016};
see \autoref{Thm_SuperRigidity}.
A concise exposition of \citeauthor{Wendl2016}’s proof using the framework of equivariant Brill--Noether theory for elliptic operators can be found in \cite{Doan2018}.

The main result of this article shows that $k$--rigidity implies a Castelnuovo bound.

\begin{theorem}
  \label{Thm_KRigidityImpliesFiniteness}
  Let $k \in \N \cup \set{\infty}$.  
  Let $(M,J,g)$ be a compact almost Hermitian $6$--manifold with a $k$--rigid almost complex structure $J$.
  Suppose $A \in H_2(M)$ satisfies $\inner{c_1(M,J)}{A} = 0$ and has divisibility at most $k$.
  Given any $\Lambda > 0$, there are only finitely many simple $J$--holomorphic maps representing $A$ and with energy at most $\Lambda$.
\end{theorem}

\begin{remark}
  \autoref{Thm_KRigidityImpliesFiniteness} immediately implies a Castelnuovo bound for every \emph{fixed} $k$--rigid almost complex structure $J$.
  Unlike in the $n>3$ case of \autoref{Eq_HighDimensionalCastelnuovo}, however, this bound may depend on $J$.
\end{remark}

If $J$ is tamed by a symplectic form $\omega$,
then imposing an upper bound for the energy is superfluous since the energy of any $J$--holomorphic map representing $A$ is $\inner{[\omega]}{A}$.

\begin{cor}
  \label{Cor_SuperRigidityImpliesFinitenessCY3}
  Let $(M,\omega)$ be a compact symplectic Calabi--Yau $3$--fold.
  Suppose $J$ is a super-rigid almost complex structure compatible with $\omega$.
  Then for every $A \in H_2(M)$ there are only finitely many simple $J$--holomorphic maps representing $A$.
\end{cor}

In the situation of \autoref{Thm_KRigidityImpliesFiniteness},
Gromov's compactness theorem \cite{Gromov1985,Parker1993,Ye1994,Hummel1997} shows that there are only finitely many $J$--holomorphic maps representing $A$ from Riemann surfaces of \emph{fixed genus}.
It is thus of no use for proving \autoref{Thm_KRigidityImpliesFiniteness}.
Instead,
we use the following compactness result for \emph{$J$--holomorphic cycles},
that is: formal sums of $J$--holomorphic curves,
with respect to \emph{geometric convergence};
see \autoref{Def_JHolomorphicCycle} and \autoref{Def_GeometricConvergence}.

\begin{prop}
  \label{Lem_JHolomorphicCycleCompactness}
  Let $M$ be a manifold and let $(J_n,g_n)_{n \in \N}$ be a sequence of almost Hermitian structures converging to an almost Hermitian structure $(J,g)$ in the $C_\loc^\infty$--topology.
  Let $K \subset M$ be a compact subset and let $\Lambda > 0$.
  For each $n \in \N$ let $C_n$ be a $J_n$--holomorphic cycle with support contained in $K$ and of mass at most $\Lambda$.
  Then a subsequence of $(C_n)_{n \in \N}$ geometrically converges to a $J$--holomorphic cycle $C$.
\end{prop}

In dimension four, this result was proved by \citet{Taubes1996}.
The proof in higher dimensions relies on results in geometric measure theory;
in particular, the recent work of \citet{DeLellis2017,DeLellis2017a,DeLellis2017b,DeLellis2017c} on the regularity of semi-calibrated currents.
The points of this theory most relevant to the present article are discussed in \autoref{Sec_RegularityTheory}.

\begin{remark}
  \label{Rem_RiviereVsDeLellis}
  If $(M,\omega)$ is a symplectic manifold, $(J_n)_{n\in\N}$ is a sequence of $\omega$--compatible almost complex structures, and $g_n = \omega(\cdot, J_n\cdot)$ is the corresponding sequence of Riemannian metrics, then \autoref{Lem_JHolomorphicCycleCompactness} can be proved using earlier work of \citet{Riviere2009} on the regularity of \emph{calibrated} currents. 
  However, the proof of \autoref{Thm_KRigidityImpliesFiniteness} leads to almost complex structures which are tamed by but (possibly) not compatible with a symplectic structure.
  Therefore, the work of \citeauthor{DeLellis2017} is crucial even for  \autoref{Cor_SuperRigidityImpliesFinitenessCY3}.
\end{remark}

\begin{remark}
  \label{Rem_GVFiniteness}
  Since the first version of this article appeared,
  we used the $k=1$ case of \autoref{Thm_KRigidityImpliesFiniteness} to prove the Gopakumar--Vafa finitness conjecture for the BPS numbers $n^g_A(M,\omega)$ whenever $A$ is a primitive homology class \cite{Doan2019}. 
  The cited article also contains a version of \autoref{Thm_KRigidityImpliesFiniteness} for homology classes satisfying $\inner{c_1(M,\omega)}{A} > 0$. 
\end{remark}

\begin{convention}
  \label{Conv_Lesssim}
  Throughout this article, $f(x) \lesssim g(x)$ is an abbreviation for:
  $f(x) \leq cg(x)$ with a constant $c>0$ independent of $x$.
\end{convention}

\paragraph{Acknowledgements}
We thank Aleksey Zinger for insightful discussions,
Tristan Rivi\'ere for answering our questions regarding \cite{Riviere2009},
Costante Bellettini for pointing us towards the work of \citeauthor{DeLellis2017a}, and
Simon Donaldson for reminding us of Gromov's h-principle for symplectic immersions.
Finally, we thank the anonymous referee for detailed comments.

This material is based upon work supported by \href{https://sloan.org/fellowships/2018-Fellows}{an Alfred P. Sloan Fellowship},
\href{https://www.nsf.gov/awardsearch/showAward?AWD_ID=1754967&HistoricalAwards=false}{the National Science Foundation under Grant No.~1754967}, and
\href{https://sites.duke.edu/scshgap/}{the Simons Collaboration ``Special Holonomy in Geometry, Analysis and Physics''}.



\section{$k$--rigidity of $J$--holomorphic maps}
\label{Sec_SuperRigidity}

Let us briefly recall the notion of $k$--rigidity as defined by
\citeauthor{Eftekhary2016}.		
For a more detailed discussion we refer the reader to
\cites[Section 2]{Eftekhary2016}[Section 2.1]{Wendl2016} as well as \cite[Section 2.1]{Doan2018}.
The notation and definitions in this article are consistent with those used in the last reference. 

Henceforth,
let $(M,J,g)$ be an almost Hermitian $2n$--manifold;
that is:
$J$ is an almost complex structure and $g$ is a Riemannian metric such that $g(J\cdot, J\cdot) = g(\cdot, \cdot)$.
In particular, we \emph{do not} assume that the $2$--form $g(J\cdot,\cdot)$ is closed or that $M$ even admits a symplectic structure.

\begin{definition}
  \label{Def_JHolomorphicMap}
  A \defined{$J$--holomorphic map} $u \co (\Sigma,j) \to (M,J)$ is a pair consisting of
  a closed, connected Riemann surface $(\Sigma,j)$ and
  a smooth map $u\co \Sigma \to M$ satisfying the non-linear Cauchy--Riemann equation
  \begin{equation}
    \label{Eq_JHolomorphic}
    \delbar_J(u,j) \coloneq \frac12(\rd u + J(u)\circ \rd u\circ  j) = 0.
    \qedhere
  \end{equation}
\end{definition}

\begin{definition}
  Let $u \co (\Sigma,j) \to (M,J)$ be a $J$--holomorphic map.
  Let $\phi \in \Diff(\Sigma)$ be a diffeomorphism.
  The \defined{reparametrization} of $u$ by $\phi$ is the $J$--holomorphic map $u \circ \phi^{-1} \co (\Sigma,\phi_*j) \to (M,J)$.
\end{definition}

\begin{definition}
  Let $u\co (\Sigma,j) \to (M,J)$ be a $J$--holomorphic map and let $\pi\co (\tilde \Sigma,\tilde j) \to (\Sigma,j)$ be a holomorphic map of degree $\deg(\pi) \geq 2$.
  The composition $u\circ \pi \co (\tilde\Sigma,\tilde j) \to (M,J)$ is said to be a \defined{multiple cover of $u$}.
  A $J$--holomorphic map is \defined{simple} if it is not constant and not a multiple cover.
\end{definition}

Rigidity and $k$--rigidity are conditions on the infinitesimal deformation theory of $J$--holomorphic curves up to reparametrization.
We will have to briefly review parts of this theory.
The reader can find further details in \cite[Chapter 3]{McDuff2012} and \cite[Lectures 2 and 7]{Wendl2016a}, for example.
The second reference, in particular, discusses varying the complex structure on higher genus Riemann surfaces.

The index of a $J$--holomorphic map $u\co (\Sigma,j) \to (M,J)$ is defined as
\begin{equation}
  \label{Eq_Index}
  \ind(u) \coloneq  2\inner{u^*c_1(M,J)}{[\Sigma]} + (2n-6)(1-g(\Sigma)).
\end{equation}
This is the Fredholm index of the linearization of \autoref{Eq_JHolomorphic} with respect to the map $u$ and complex structure $j$ (up to equivalence).
The linearization with respect to $u$, with $j$ fixed, is the operator
\begin{equation}
  \label{Eq_DU}
  \xi
  \mapsto
  \frac12\paren*{\nabla \xi + J\circ (\nabla \xi)\circ j + (\nabla_\xi J)\circ \rd u\circ j}.
\end{equation}
Here $\nabla$ denotes the Levi--Civita connection of $g$ on $TM$ and also the induced connection on $u^*TM$  \cite[Proposition 3.1.1]{McDuff2012}.

Let $u\co (\Sigma,j) \to (M,J)$ be a non-constant $J$--holomorphic map.
There exists a unique complex subbundle 
\begin{equation*}
  Tu \subset u^*TM
\end{equation*}
of rank one containing $\rd u(T\Sigma)$ 
 \cites[Section 1.3]{Ivashkovich1999}.
The generalized normal bundle of $u$ is defined as
\begin{equation*}
  Nu \coloneq u^*TM/Tu.
\end{equation*}
If $u$ is an immersion, then $Nu$ is the usual normal bundle.
If $\tilde u = u\circ \pi$ is a multiple cover of an immersion, then $N\tilde u = \pi^*Nu$.
The operator \eqref{Eq_DU} maps $\Gamma(Tu)$ to $\Omega^{0,1}(\Sigma,Tu)$.
Thus, it induces an operator
\begin{equation}
  \label{Eq_NormaCauchyRiemann}
  \fd_{u,J}^N \co \Gamma(Nu) \to \Omega^{0,1}(Nu)
\end{equation}
called the \defined{normal Cauchy--Riemann operator of $u$} \cite[(1.5.1)]{Ivashkovich1999}.
The non-zero elements of the kernel of $\fd_{u,J}^N$ correspond to infinitesimal deformations of $u$ which deform the image $u(\Sigma)$.
The reader might find the summaries of \citeauthor{Ivashkovich1999}'s construction of $Tu$, $Nu$, and $\fd_{u,J}^N$  given in \cites[Section 3.3]{Wendl2010}[Appendix 2A]{Doan2018} helpful.

\begin{definition}
  \label{Def_RigidMap}
  A non-constant $J$--holomorphic map $u$ is \defined{rigid} if $\ker \fd_{u,J}^N = 0$.
\end{definition}

A multiple cover $\tilde u$ of $u$ may fail to be rigid,
even if $u$ itself is rigid.

\begin{definition}
  \label{Def_KRigidMap}
  Let $k \in \N \cup \set{\infty}$.
  A simple $J$--holomorphic map $u\co (\Sigma,j) \to (M,J)$ is called \defined{$k$--rigid} if it is rigid and all of its multiple covers of degree at most $k$ are rigid.
\end{definition}

Rigidity and $k$--rigidity are mostly interesting for maps of index zero, as it follows from the index formula for the normal Cauchy--Riemann operator \cite[Lemma 1.5.1]{Ivashkovich1999} (see also \cites[Theorem 3]{Wendl2010}[Proposition 2.7.1]{Doan2018})  and standard transversality results that for a generic $J$ there are no rigid simple $J$--holomorphic maps satisfying $\ind(u) \neq 0$.  

\begin{definition}
  \label{Def_KRigidAlmostComplexStructure}
  Suppose that $\dim M \geq 6$. 
  Let $k \in \N \cup \set{\infty}$.
  An almost complex structure $J$ is called \defined{$k$--rigid} if the following hold:
  \begin{enumerate}
  \item
    \label{Def_KRigidAlmostComplexStructure_KRigid}
    Every simple $J$--holomorphic map of index zero is $k$--rigid.
  \item
    \label{Def_KRigidAlmostComplexStructure_NonNegativeIndex}
    Every simple $J$--holomorphic map has non-negative index.
  \item
    \label{Def_KRigidAlmostComplexStructure_Embedding}
    Every simple $J$--holomorphic map of index zero is an embedding, and
    every two simple $J$--holomorphic maps of index zero either have disjoint images or are related by a reparametrization.
    \qedhere
  \end{enumerate}
\end{definition}

\begin{remark}
  In dimension four, one should weaken \autoref{Def_KRigidAlmostComplexStructure_Embedding} and require only that every simple $J$--holomorphic map of index zero is an immersion with transverse self-intersections, and that two such maps are either transverse to one another or are related by reparametrization.
  However, we will only be concerned with dimension (at least) six.
\end{remark}

\begin{definition}
  \label{Def_JKRigid}
  Denote by $\sJ(M)$ the Fr\'echet space of smooth almost complex structures on $M$ equipped with the topology of $C^\infty$ convergence over compact subsets.
  If $\omega$ is a symplectic form on $M$, denote  by $\sJ(M,\omega)$ the subspace of almost complex structures in $\sJ(M)$ compatible with $\omega$. 
  For $k \in \N \cup \set{\infty}$ set
  \begin{equation*}
    \sR_k(M)
    \coloneq
    \set*{
      J \in \sJ(M)
      :
      J
      \textnormal{ is $k$--rigid}
    }
    \qandq
    \sR_k(M,\omega)
    \coloneq
    \sR_k(M) \cap \sJ(M,\omega).
    \qedhere
  \end{equation*}
\end{definition}

\begin{theorem}[{\citet[Theorem A]{Wendl2016}}]
  \label{Thm_SuperRigidity}
  Let $M$ be a manifold of dimension at least six.
  The following hold:
  \begin{enumerate}
    \item The subset $\sR_\infty(M)$ is comeager in $\sJ(M)$.
    \item The subset $\sR_\infty(M,\omega)$ is comeager in $\sJ(M,\omega)$. 
  \end{enumerate}
\end{theorem}

This result establishes the \emph{super-rigidity conjecture} of \citet[p.~290]{Bryan2001}.
Earlier progress on this conjecture was made by \citet[Theorem 1.2]{Eftekhary2016} who showed that $\sR_4(M,\omega)$ is comeager in $\sJ(M,\omega)$ if $\dim M = 6$.



\section{Real Cauchy--Riemann operators and almost complex structures}
\label{Sec_AlmostComplexGeometry}

The purpose of this section is to explain that associated with every real Cauchy--Riemann operator defined on a Hermitian vector bundle there is a natural almost complex structure on the total space of that bundle.
This construction is inspired by \cite[p. 825--826]{Taubes1996b}.
In fact,
the results below can be found in \cites[Section 3.2]{Zinger2011}[Appendix B]{Wendl2016}.
Nevertheless,
we include them here for the reader's convenience.

\begin{definition}
  \label{Def_RealCauchyRiemann}
  Let $(\Sigma,j)$ be a Riemann surface.
  Let $\pi\co E \to \Sigma$ be a Hermitian vector bundle over $\Sigma$.
  A first order linear differential operator $\fd \co \Gamma(E) \to \Omega^{0,1}(\Sigma,E)$ is called a \defined{real Cauchy--Riemann operator}
  if
  \begin{equation}
    \label{Eq_RealCauchyRiemann}
    \fd(fs) = (\delbar f)s + i\fd j
  \end{equation}  
  for all $f \in C^\infty(M,\R)$.
  The \defined{anti-linear part} of $\fd$ is defined as
  \begin{equation*}
    \fn = \fn_\fd \coloneq \frac12(\fd + J\fd J) \in \Gamma(\Hom(E,\overline\Hom_\C(T\Sigma,E))).
    \qedhere.
  \end{equation*}
\end{definition}

Every real Cauchy--Riemann operator can be written as
\begin{equation*}
  \fd = \delbar + \fn
\end{equation*}
where $\delbar_\nabla \coloneq \nabla^{0,1}$ is the Dolbeault operator associated with a Hermitian connection $\nabla$ on $E$.
Denote by $H_{\nabla} \subset TE$ the horizontal distribution of $\nabla$.
It induces an isomorphism
\begin{equation}
  \label{Eq_HorizontalVertical}
  TE
  = H_{\nabla} \oplus \pi^* E
  \iso \pi^* T\Sigma \oplus \pi^*E.
\end{equation}

\begin{definition}
  \label{Def_JNabla}
  The \defined{complex structure} $J_\nabla$ on $E$ associated with $\nabla$ is defined by pulling back the standard complex structure $j\oplus i$ on $\pi^* T\Sigma \oplus \pi^*E$ by the isomorphism \autoref{Eq_HorizontalVertical}.
\end{definition}

It is well-known that a section $s \in \Gamma(E)$ satisfies $\delbar_\nabla s = 0$ if and only if the map $s \co \Sigma \to E$ is $J_\nabla$--holomorphic.
The following proposition extends this to real Cauchy--Riemann operators.

\begin{definition}
  \label{Def_JD}
  Let $\fd = \delbar_\nabla + \fn$ be a real Cauchy--Riemann operator.
  Define $L_\fn \co TE \to TE$ by
  \begin{equation*}
    L_\fn = - 2\fn(v) j \pi_*
  \end{equation*}
  at $v \in E$.
  The \defined{almost complex structure} $J_\fd$ on $E$ associated with $\fd$ is defined by
  \begin{equation*}
    J_\fd \coloneq J_{\nabla} + L_\fn.
    \qedhere
  \end{equation*}
\end{definition}

\begin{lemma}
  \label{Prop_PropertiesOfJD}
  For every real Cauchy--Riemann operator $\fd \co \Gamma(E) \to \Omega^{0,1}(E)$ the following hold:
  \begin{enumerate}
  \item
    \label{Prop_PropertiesOfJD_AlmostComplex}
    $J_\fd$ is an almost complex structure.
  \item
    \label{Prop_PropertiesOfJD_Projection}
    The projection $\pi \colon E \to \Sigma$ is holomorphic with respect to $J_\fd$.
  \item
    \label{Prop_PropertiesOfJD_Fiber}
    For every $x \in \Sigma$ the fiber $E_x = \pi^{-1}(x)$ is a $J_\fd$--holomorphic submanifold of $E$.
  \item
    \label{Prop_PropertiesOfJDA_Section}
    A section $s \in \Gamma(E)$ satisfies $\fd s = 0$ if and only if $s \co \Sigma \to E$ is a $J_\fd$--holomorphic map.
  \item
    \label{Prop_PropertiesOfJD_Tamed}
    Denote by $B_1(E) \coloneq \set{ e \in E : \abs{e} < 1}$ the disc bundle of $E$ with respect to the given Hermitian inner product on $E$.  
    There exists a symplectic form $\omega$ on the total space of $B_1(E)$ which tames $J_\fd$.
  \end{enumerate}
\end{lemma}

\begin{proof}
  With respect to \autoref{Eq_HorizontalVertical} we have
  \begin{equation}
    \label{Eq_JDMatrix}
    J_\fd =
    \begin{pmatrix}
      j & 0 \\
      -2\fn(v)j & i
    \end{pmatrix} 
  \end{equation}
  at $v \in E$.
  Since $\fn(v)$ is anti-linear,
  \begin{equation*}
    \fn(v)j^2 + i\fn(v)j = 0.
  \end{equation*}
  Therefore,
  \begin{equation*}
    J_\fd^2 = -\id;
  \end{equation*}
  that is, \autoref{Prop_PropertiesOfJD_AlmostComplex} holds.

  Both \autoref{Prop_PropertiesOfJD_Projection} and \autoref{Prop_PropertiesOfJD_Fiber} immediately follow from \eqref{Eq_JDMatrix}.

  We prove \autoref{Prop_PropertiesOfJDA_Section}.
  Let $s \colon \Sigma \to E$ be a section.
  The projection of $\rd s$ to the first factor of \autoref{Eq_HorizontalVertical} is $\pi_* \circ \rd s = \id_{T\Sigma}$ and thus $j$--linear.
  The projection of $\rd s \colon T\Sigma \to s^*TE$ to the second factor is its covariant derivative $\nabla s \colon T\Sigma \to s^*E$.
  It follows from \autoref{Eq_JDMatrix} that the $J_\fd$--antilinear part of $\rd s$ is
  \begin{align*}
    \frac12(\rd s + J_\fd \circ \rd s \circ j)
     &=
     \frac12( \nabla s + i \circ \nabla s \circ j) + \fn s \\
     &=
      \delbar_{\nabla}s + \fn s = \fd s.
  \end{align*}
  Therefore,
  $\rd s \colon T\Sigma \to TE$ is $J_\fd$--linear if and only if $\fd s = 0$.
  
  The construction of a symplectic form $\omega$ in \autoref{Prop_PropertiesOfJD_Tamed} is standard and goes back to \citet{Thurston1976};  see also \cites[Lemma 2.2]{Gompf1995}[Theorem 6.3]{McDuff1998}[paragraph containing (2.9)]{Tehrani2018}.
  Nevertheless, let us discuss the proof of \autoref{Prop_PropertiesOfJD_Tamed}.
  Let $\omega_\Sigma$ be an area form on $\Sigma$.
  Let $\omega_E$ be any closed $2$--form on $B_1(E)$ which is positive when restricted to the fibers of $E$;
  that is, for all vertical tangent vectors $v_E$ 
  \begin{equation}
    \label{Eq_FiberwisePositivity}
    \omega_E(v_E, J_\nabla v_E) \gtrsim \abs{v_E}^2.
  \end{equation}
  Such a form can be constructed by choosing local unitary trivializations $E|_{U_i} \iso U_i \times \C^r$, denoting by $\lambda_i$ the corresponding Liouville $1$--forms on $\C^r $ vanishing at zero, and setting
  \begin{equation*}
    \omega_E = \rd\(\sum_i \chi_i\circ\pi \cdot \lambda_i\)
  \end{equation*}
  for a partition of unity $(\chi_i)$.
  This form satisfies \autoref{Eq_FiberwisePositivity} on $E$.
  It remains to show that for $\tau \gg 1$ the closed $2$--form $\omega = \tau\omega_\Sigma + \omega_E$ tames $J_\fd$ on $B_1(E)$.
  For a tangent vector $w$ to $E$ at a point $(x,v) \in B_1(E)$ denote by $w_H$ and $w_E$ its horizontal and vertical parts in the decomposition \autoref{Eq_HorizontalVertical}.
  We have
  \begin{align*}
    \omega(w, J_\fd w)
    &=
      (\tau\omega_\Sigma + \omega_E)(w, (J_\nabla + L_\fn) w) \\
    &=
      \tau\omega_\Sigma(w_H,jw_H) + \omega_E(w_E, J_\nabla w_E) + \omega_E(w_E, L_\fn w_H).
  \end{align*}
  From $\abs{L_\fn(v)} \lesssim \abs{v} < 1$ it follows that
  \begin{equation*}
    \abs{\omega_E(w_E, L_\fn w_H)}
    \lesssim
    \abs{w_E}\abs{w_H}.
  \end{equation*}
  Since
  \begin{equation*}
    \tau\omega_\Sigma(w_H,jw_H) + \omega_E(w_E, J_\nabla w_E)
    \gtrsim
    \tau\abs{w_H}^2 + \abs{v_E},
  \end{equation*}
  it follows that $\omega$ tames $J_\fd$ provided $\tau \gg 1$.
\end{proof}

The next two propositions are concerned with the following situation.
Let $(M,J,g)$ be an almost Hermitian manifold and let $u \co (\Sigma,j) \to (M,J)$ be a $J$--holomorphic embedding.
Denote by $Nu \to \Sigma$ its normal bundle and by $\fd_{u,J}^N$ the normal Cauchy--Riemann operator introduced in \eqref{Eq_NormaCauchyRiemann}.
The almost complex structure $J$ and Riemannian metric $g$ on $M$ induce a Hermitian structure on $E$. 
Write
\begin{equation}
  \label{Eq_JU}
  J_u \coloneq J_{\fd_{u,J}^N}
\end{equation}
for the almost complex structure on the total space of $Nu$ associated with $\fd_{u,J}^N$.

\begin{lemma}
  \label{Prop_RescalingComplexStructure}
  For every $\lambda > 0$ define $\sigma_\lambda\co Nu \to Nu$ by
  \begin{equation*}
    \sigma_\lambda(v) \coloneq \lambda v.
  \end{equation*}
  If $U \subset Nu$ is an open neighborhood of the zero section in $Nu$ such that the exponential map $\exp \co U \to M$ with respect to  $g$ is an embedding,
  then
  \begin{equation*}
    \sigma_\lambda^*\exp^*J \to J_u
    \quad\text{as}\quad
    \lambda \to 0,
  \end{equation*}
  where the convergence is with respect to the $C_\loc^\infty$--topology.
\end{lemma}

\begin{proof} 
  Denote by $\nabla$ the connection on $Nu \to \Sigma$ induced by the Levi--Civita connection of $(M,g)$. 
  Throughout this proof,
  we identify
  \begin{equation*}
    TU = \pi^*T\Sigma \oplus \pi^*Nu
  \end{equation*}
  as in \autoref{Eq_HorizontalVertical}.  
  The two almost complex structures $J_{\nabla}$ and $\exp^*J$ on $U \subset Nu$ agree along the zero section.
  The Taylor expansion of $\exp^*J$ is of the form
  \begin{equation}
    \label{Eq_TaylorExpansion}
    \exp^*J(x,v) = J_\nabla(x,0) + \nabla_v J(x,0) + O(\abs{v}^2).
  \end{equation}
  
  Set
  \begin{equation*}
    L(x,v) \coloneq \nabla_n J(x,0).
  \end{equation*}  
  We write $L$ as the matrix 
  \begin{equation*}
    L(x,v) = 
    \begin{pmatrix}
      L_{11}(x,v) & L_{12}(x,v) \\
      L_{21}(x,v) & L_{22}(x,v)
    \end{pmatrix}.
  \end{equation*}
  Here each $L_{ij}$ is linear in $v$.
  The derivative $\rd \sigma_{\lambda}$ is given by
  \begin{equation*}
    \rd \sigma_{\lambda}
    = 
    \begin{pmatrix}
      \id & \\
      & \lambda
    \end{pmatrix}.
  \end{equation*}
  Therefore,
  \begin{align*}
    (\sigma_\lambda)^*L(x,v)
    &=
      \begin{pmatrix}
        \id & \\
        & \lambda^{-1}
      \end{pmatrix}
      \begin{pmatrix}
        L_{11}(x,\lambda v) & L_{12}(x,\lambda v) \\
        L_{21}(x,\lambda v) & L_{22}(x,\lambda v)
      \end{pmatrix}
      \begin{pmatrix}
        \id & \\
        & \lambda
      \end{pmatrix} \\
    &=
      \begin{pmatrix}
        \lambda L_{11}(x,v) & \lambda^2 L_{12}(x,v) \\
        L_{21}(x,v) & \lambda L_{22}(x, v)
      \end{pmatrix}.
  \end{align*}
  As $\lambda$ tend to zero, all but the bottom left entry tend to zero.

  By construction, $\sigma_\lambda^* J_\nabla = J_\nabla$.
  As $\lambda$ tends to zero,
  the rescalings of terms of second order and higher in \autoref{Eq_TaylorExpansion} tend to zero.
  It remains to identify the term $L_{21}$.
  By definition,
  \begin{equation*}
    L_{21}(x,v) = \pi_{Nu} \circ \nabla_v J(x,0) \circ \pi_*. 
  \end{equation*}
  Comparing \eqref{Eq_DU}, \autoref{Def_RealCauchyRiemann}, and  \autoref{Def_JD},
  we see that $L_{21} = L_u$.
  This finishes the proof.
\end{proof}

\begin{prop}
  \label{Prop_CurvesInNormalBundle}
  If $\tilde u \colon (\tilde\Sigma, \tilde j) \to (Nu, J_u)$ is a simple $J_u$--holomorphic map  whose image is not contained in the zero section,
  then the following hold:
  \begin{enumerate}
  \item
    \label{Prop_CurvesInNormalBundle_VarphiHolomorphic}
    The map $\varphi \colon (\tilde\Sigma, \tilde j) \to (\Sigma,j)$ given by $\varphi \coloneq \pi \circ \tilde u$ is non-constant and holomorphic.
  \item
    \label{Prop_CurvesInNormalBundle_NotSuperRigid}
    The $J$--holomorphic map $u \circ \varphi \colon (\tilde\Sigma,\tilde j) \to (M,J)$ is not rigid;
    in particular, the $J$--holomorphic map $u \colon (\Sigma,j) \to (M,J)$ is not $k$--rigid for $k = \deg(\varphi)$.
  \end{enumerate}
\end{prop}

\begin{proof}  
  By \autoref{Prop_PropertiesOfJD}~\autoref{Prop_PropertiesOfJD_Projection},
  $\pi \colon Nu \to \Sigma$ is $J_u$--holomorphic.
  Therefore, $\varphi$ is holomorphic.
  The map $\varphi$ is constant if and only if the image of $\tilde u$ is contained in a fiber of $\pi$.
  This is impossible, because then $\tilde u$ would be constant.
  This proves \autoref{Prop_CurvesInNormalBundle_VarphiHolomorphic}.

  To prove \autoref{Prop_CurvesInNormalBundle_NotSuperRigid}, we use the fact that the normal bundle of the $J$--holomorphic map $u \circ \varphi$ is $N{u \circ \varphi} = \varphi^* Nu$ and the corresponding normal Cauchy--Riemann operator is
  \begin{equation}
    \label{Eq_DNPullBack}
    \fd_{u \circ \varphi, J}^N = \varphi^* \fd_{u,J}^N.
  \end{equation}
  Since $\tilde u$ takes values in $Nu$, for every $x \in \tilde\Sigma$ we have
  \begin{equation*}
    \tilde u(x) \in Nu_{\pi(u(x))} = Nu_{\varphi(x)} = (\varphi^*Nu)_x.
  \end{equation*} 
  This gives rise to the section $s \in \Gamma(\varphi^* Nu)$ defined by
  \begin{equation*}
    s(x) \coloneq \tilde u(x) \in (\varphi^* Nu)_x.
  \end{equation*}  
  This section is not the zero section, because the image of $\tilde u$ is not contained in the zero section.
  By construction,  $s$ is holomorphic with respect to the almost complex structure on $\phi^*Nu$ induced from $J_u$.
  \autoref{Prop_PropertiesOfJD}~\autoref{Prop_PropertiesOfJDA_Section} and \autoref{Eq_DNPullBack} imply that
  \begin{equation*}
    \fd_{u\circ\varphi,J}^Ns = 0. \qedhere
  \end{equation*}
\end{proof}



\section{Regularity theory for $2$--dimensional semicalibrated currents}
\label{Sec_RegularityTheory}

The purpose of this section is to introduce a few notions of geometric measure theory and explain \autoref{Lem_SingularityStructure} due to \citeauthor{DeLellis2017}.
The standard reference for the foundations of geometric measure theory is \citeauthor{Federer1969}'s voluminous monograph \cite{Federer1969}.
The references \cite{Simon1983,DeLellis2016} are more accessible and easier to navigate,
and are cited throughout this section.

\begin{definition}[{cf. \cites[Chapter 6 Definition 26.1 and Paragraphs 26.3, 26.10, 26.11]{Simon1983}[Definitions 2.1, 2.2]{DeLellis2016}}]
  Let $M$ be a manifold. 
  Let $\Omega_c^k(M)$ be the space of $k$--forms with compact support equipped with the strong $C^\infty$ topology.%
  \footnote{For the definition, see, for example, \cite[Chapter II Section 3]{Golubitsky1980}. 
    A sequence $(\alpha_n)_{n\in\N}$ in $\Omega_c^k(M)$ converges in this topology if and only if there is a compact subset $K \subset M$ such that $\supp \alpha_n \subset K$ for all sufficiently large $n$ and $(\alpha_n)_{n\in\N}$ converges uniformly with all derivatives over $K$.  
    If $M$ is compact, then the strong $C^\infty$ topology agrees with the standard $C^\infty$ topology making $\Omega^k(M)$ into a Fr\'echet space.
    If $M$ is non-compact, then the strong $C^\infty$ topology is not metrizable.
  }
  \begin{enumerate}
  \item
    A \defined{$k$--current} in $M$ is a continuous linear map $T \co \Omega^k_c(M) \to \R$.
  \item 
    A sequence $(T_n)_{n\in\N}$ of $k$--currents \defined{converges weakly} to a $k$--curent $ T$ if
    \begin{equation*}
      \lim_{n\to\infty} T_n(\alpha) = T(\alpha) \qforeveryq \alpha \in \Omega^k_c(M).
    \end{equation*}
  \item
    The \defined{boundary} of a $k$--current $T$ is the $(k-1)$--current $\partial T$ defined by
    \begin{equation*}
      \partial T(\alpha) \coloneq  T(\rd \alpha)
    \end{equation*}
    We say that $ T$ is \defined{closed} if $\partial T = 0$. 
  \item
    The \defined{support} of a $k$--current $T$, denoted by $\supp(T)$, is the intersection of all closed subsets $A \subset M$ with the property that  $T(\alpha) = 0$ for all $\alpha \in \Omega^k_c(M)$ with $\supp \alpha \cap A = \emptyset$.
  \item
    Given an open subset $U \subset M$ and a $k$--current $T$, the \defined{restriction of $T$ to $U$} is the $k$--current $T|_U \co \Omega^k_c(U) \to \R$ defined by
    \begin{equation*}
      T|_U(\alpha) \coloneq T(\iota_*\alpha)
    \end{equation*}
    with $\iota_* \co \Omega^k_c(U) \to \Omega^k_c(M)$ denoting the map extending a $k$--form with compact support in $U$ by zero on $M \setminus U$.
    \qedhere    
  \end{enumerate}  
\end{definition}  

The archetypal example of a $k$--current is the following.

\begin{example}
  \label{Ex_DiracDelta}
  Let 
  \begin{equation*}
    A = \sum_{i=1}^I m_i A_i
  \end{equation*}
  be a formal linear combination of oriented $k$--dimensional $C^1$ submanifolds $A_i \subset M$ possibly with non-empty boundary $\partial A_i$ and with coefficients $m_1, \ldots, m_I \in \N$. 
  The \defined{Dirac delta associated with $A$} is the $k$--current $\delta_A \co \Omega^k_c(M) \to \R$ defined by
  \begin{equation*}
    \delta_A(\alpha)
    \coloneq
    \sum_{i=1}^I m_i \int_{A_i} \alpha.
  \end{equation*} 
  We have $\supp \delta_A = \bigcup_{i=1}^I A_i$. 
  The boundary of $A$ is the Dirac delta associated with the formal sum
  \begin{equation*}
    \partial A = \sum_{i=1}^I m_i \partial A_i.
  \end{equation*}
  More generally, $\delta_A$ can be defined in the same way if each $A_i$ is an oriented $k$--dimensional $C^1$ submanifolds away from a subset whose $k$--dimensional Hausdorff measure is zero. 
\end{example}

\begin{definition}[{cf. \cite[Definition 3.2]{DeLellis2016}}]
  Let $T$ be a $k$--current in $M$.
  A point $x \in \supp(T) \setminus \supp (\partial T)$ is \defined{regular} if there exists an open neighborhood $U$ of $x$ such that $T|_U = \delta_{mA}$ for an oriented $k$--dimensional $C^1$ submanifold $A \subset U$ and $m \in \N$. 
  Otherwise we say that $x$ is \defined{singular}.
  Denote by $\reg(T)$ and $\sing(T) $ the sets of regular and singular points in $\supp(T) \setminus \supp (\partial T)$.
\end{definition}

A generalization of the above example is the notion of an integral current, which we now describe. 
For the remainder of this section, let $(M,g)$ be a Riemannian manifold and let $\sH^k$ denote the $k$--dimensional Hausdorff measure induced by $g$.

\begin{remark}
  \label{Rem_GMTRiemannianManifolds}
  While the definitions and results of geometric measure theory given below are typically stated for $M = \R^N$ equipped with the Euclidean metric, they immediately generalize to any Riemannian manifold $(M,g)$ by embedding it isometrically into $\R^N$ for some $N$ using the Nash embedding theorem, and considering currents in $M$ as currents in $\R^N$.
\end{remark}

\begin{definition}[{cf. \cites[Chapter 6 Paragraphs 25.6, 26.4]{Simon1983}[Definitions 2.3, 3.4]{DeLellis2016}}]
  For $x\in M$ and $\alpha_x \in \Lambda^k T_x^*M$, set $\abs{\alpha_x} \coloneq \sup\set{ \abs{\alpha_x(\xi)} }$ with the supremum taken over all simple $k$--vectors $\xi \in \Lambda^k T_x M$ with $\abs{\xi} = 1$, with the norm induced by the Riemannian metric. 
  The \defined{comass} of $\alpha \in \Omega^k_c(M)$ is
  \begin{equation*}
    \Abs{\alpha} \coloneq \sup_{x\in M}\abs{\alpha_x}.
  \end{equation*}  
  The \defined{mass} of a $k$--current $T$ in $M$ is defined by
  \begin{equation*}
    \bM(T)
    \coloneq
    \sup\set{
      T(\alpha)
      :
      \alpha \in \Omega^k_c(M) \textnormal{ and } \Abs{\alpha} \leq 1
    }.
    \qedhere
  \end{equation*}
\end{definition}

\begin{definition}[{cf. \cite[Chapter 3 Section 11]{Simon1983}}]
  A subset $A \subset M$ is called  \defined{$k$--rectifiable} if there are subsets $A_0, A_1, A_2, \ldots \subset M$ such that $\sH^k(A_0) = 0$, each $A_i$ for $i\geq 1$ is a $C^1$--embedded $k$--dimensional submanifold, and
  \begin{equation*}
    A \subset \bigcup_{i=0}^\infty A_i.
  \end{equation*}
\end{definition}

If $A \subset M$ is $k$--rectifiable, then for $\sH^k$--almost every $x \in A$ there exists an \defined{approximate tangent space} to $A$ at $x$, which is a $k$--dimensional subspace of $T_x M$.
We denote it by $\pi(A,x)$.
For details, see \cites[Chapter 3 Section 11]{Simon1983}[Lemma 2.1.15]{DeLellis2016},

\begin{definition}[{cf. \cite[Definition 27.1]{Simon1983}}]
  A $k$--current $T$ in $M$ is \defined{integer rectifiable} if $\bM(T) < \infty$ and there exist:
  \begin{enumerate}
  \item
    a $k$--rectifiable subset $A \subset M$,
  \item
    an $\sH^k$--measurable function $m \co A \to \N\cup\set{0}$, and
  \item
    an $\sH^k$--measurable section $\overrightarrow{T}$ of  $\Lambda^k TM|_A$
  \end{enumerate}
  such that:
  \begin{enumerate}[resume]
  \item
    for $\sH^k$--almost all $x \in A$, the $k$--vector $\overrightarrow{T}(x) \in \Lambda^k T_x M$ is given by $\overrightarrow{T}(x) = e_1 \wedge \ldots \wedge e_k$ for any orthonormal frame of $\pi(A,x)$, and
  \item
    $T$ is given by
    \begin{equation*}
      T(\alpha) = \int_A  m(x) \Inner{\alpha(x), \overrightarrow{T}(x)} \, \rd\sH^k
      \qforq \alpha \in \Omega^k_c(M).
    \end{equation*}
    Here $\Inner{\cdot,\cdot}$ is the pairing between $k$--forms and $k$--vectors and the integral is taken with respect to the $k$--dimensional Hausdorff measure. 
  \end{enumerate}  
  We say that $T$ is \defined{integral} if both $T$ and $\partial T$ are integer rectifiable. 
  In particular, a closed integer rectifiable current is integral.
\end{definition}

\begin{definition}[{cf. \cites[Definition 5.5]{DeLellis2016}[Definition 0.1(b)]{DeLellis2017b}}]
  A $k$--\defined{semicalibration} on $M$ is a $k$--form $\sigma$ such that $\Abs{\sigma} \leq 1$.
  An integral $k$--current $T$ in $M$ is \defined{semicalibrated} by $\sigma$ if $\sigma_x(\overrightarrow{T}(x)) = 1$ for $\sH^k$--almost every $x \in \supp T$.
\end{definition}

\begin{remark}
  \label{Rem_Calibrations}
  A semicalibration $\sigma$ is called a \defined{calibration} if $\rd\sigma =0$. 
  This notion was introduced in a seminal article by \citet{Harvey1982},
  who observed that a calibrated current is volume-minimizing. 
\end{remark}

\begin{theorem}[Federer--Fleming Compactness Theorem]
  \label{Thm_FedererFleming}
  Let $(T_n)_{n \in \N}$ be a sequence of integral $k$--currents in $M$ and let $K \subset M$ be a compact subset.
  If 
  \begin{equation*}
    \sup_n\set{\bM(T_n) + \bM(\partial T_n)} < \infty \qandq \supp(T_n) \subset K
    \qforeveryq
    n \in \N,
  \end{equation*}
  then after passing to a subsequence $(T_n)_{n\in\N}$ converges weakly to an integral current $T$.
  Moreover, if each $T_n$ is closed, then so is $T$. 
  If each $T_n$ is semicalibrated by a semicalibration $\sigma$, then so is $T$. 
\end{theorem}

The proof of the first two statements can be found, for example, in \cite[Chapter 6 Section 32]{Simon1983}.
The statement about semicalibrations follows immediately: an integral current $T$ with $\supp T \subset K$ is semicalibrated by $\sigma$ if and only if $T(\chi \sigma) = \vol(T)$, for any $\chi \in C^\infty_c(M)$ with $\chi = 1$ on $K$, and this condition is preserved by the weak convergence.

The upcoming discussion requires the following notation and definition.

\begin{notation}
  \label{Not_KBranching}
  Given $k \in \N$,
  set
  \begin{equation*}
    \tilde D^k
    \coloneq
    \set{
      (z,w) \in \C^2 : z = w^k \text{ and } \abs{z} < 1
    }. 
  \end{equation*}
  $\tilde D^k\setminus\set{0}$ is an oriented smooth submanifold of $\C^2$ and the map $(z,w) \mapsto z$ is an orientation-preserving local diffeomorphism $\tilde D^k\setminus\set{0} \to D\setminus\set{0}$ where $D \coloneq \set{ z \in \C : \abs{z} < 1}$.
  We equip $\tilde D^k\setminus\set{0}$ with the pull-back of the flat metric on $D\setminus\set{0}$.%
  \footnote{While $\tilde D^k$ is homeomorphic to $D$ and thus has the structure of a smooth manifold, the pull-back metric on $\tilde D^k\setminus\set{0}$ does not extend through the origin unless $k=1$.}

  Let $k \in \N$ and $\alpha \in (0,1)$.
  Let $f \co \tilde D^k \to \R^{2n-2}$ be a continuous injective map  which is of class $C^{3,\alpha}$ on $\tilde D^k\setminus\set{0}$ and satisfies $f(0) = 0$ and $\abs{\rd f(z,w)} \lesssim \abs{z}^\alpha$.
  Define $\underline{f} \co \tilde D^k \to \R^{2n}$ by
  \begin{equation*}
    \underline{f}(z,w) \coloneq (z, f(z,w)).
    \qedhere
  \end{equation*}
\end{notation}

\begin{definition}[{cf. \cite[Definition 1.3]{DeLellis2017b}}]
  \label{Def_KBranching}
  Let $U \subset \R^{2n}$ be an open neighborhood of zero and let $\phi \co U \to M$ be a smooth chart.
  Given $f$ and $\phi$ as in \autoref{Not_KBranching}   the \defined{$k$--branching associated with $(f,\phi)$} is the integral $2$--current $G_{f,\phi}$ on $M$ given by
  \begin{gather*}
    G_{f,\phi} \co \Omega^2_c(M) \to \R \\
    G_{f,\phi}(\alpha)
    \coloneq
    \int_{\tilde D^k\setminus\set{0}} \underline{f}^*\phi^*\alpha
    \qforq
    \alpha \in \Omega^2_c(M).
    \qedhere
  \end{gather*}
\end{definition}

We are ready to state the crucial regularity theorem for $2$--dimensional semicalibrated currents used in this paper.

\begin{theorem}[{\citet{DeLellis2017b,DeLellis2017}}]
  \label{Lem_SingularityStructure}
  Let $\sigma$ be a $2$--semicalibration on a Riemannian manifold $(M,g)$ and let $T$ be an integral $2$--current in $M$ semicalibrated by $\sigma$.
  \begin{enumerate}
  \item
    \label{Lem_SingularityStructure_LocalPresentation}
    For every $x \in \supp(T) \setminus \supp(\partial T)$ there exist a neighborhood $U$ of $x$, a finite collection of maps $f_1,\ldots,f_I$ and charts $\phi_1,\ldots,\phi_I$ as in \autoref{Not_KBranching} with $\phi_i(0) = x$, and weights $m_1,\ldots,m_I \in \N$ such that
    \begin{equation}
      \label{Eq_LocalPresentation}
      T|_U = \sum_{i=1}^I m_i G_{f_i,\phi_i}.
    \end{equation}
  \item 
    \label{Lem_SingularityStructure_SingularSet}
    The set $\sing(T)$ is discrete.
  \end{enumerate}
\end{theorem}

\begin{proof}[Discussion of the proof]
  This result, for $M=\R^N$ with the Euclidean metric, is contained in \cite[Theorem 0.2 and Section 1]{DeLellis2017b} and \cite[Theorem 3.1, Section 3.2]{DeLellis2017}.
  The result for an arbitrary $M$ follows by the argument explained in \autoref{Rem_GMTRiemannianManifolds}. 
  Although only part \autoref{Lem_SingularityStructure_SingularSet} of the theorem is explicitly stated in the article \cite{DeLellis2017b}, its authors actually prove \autoref{Lem_SingularityStructure_LocalPresentation} which is a slightly stronger statement.
  Indeed, \autoref{Lem_SingularityStructure_LocalPresentation} implies \autoref{Lem_SingularityStructure_SingularSet} because $x$ is the only singular point of any current of the form \autoref{Eq_LocalPresentation}.
  For the reader's convenience, we outline how to reconstruct the proof of \autoref{Lem_SingularityStructure_LocalPresentation} from the discussion in \cite[Section 1; especially 1.4---1.5]{DeLellis2017b}.
  
  The first step in the proof, explained in \cite[Step 4 in Section 3.2]{DeLellis2017}, is to show that without loss of generality we may assume that  $T$ is \defined{irreducible} in the sense that it cannot be written in the form $T = T_1 + T_2$ for two integral currents $T_1$, $T_2$ with $\supp T_1 \cap \supp T_2 = \emptyset$. 
  For $\lambda > 0$ denote by $T_\lambda$ the rescaled current
  \begin{equation*}
    T_\lambda(\alpha) \coloneq T(\lambda^*\alpha) \qforq \alpha \in \Omega^2_c(\R^N),
  \end{equation*}
  where, by slight abuse of notation, $\lambda^*\alpha$ denotes the pull-back of $\alpha$ by the diffeomorphism $y \mapsto \lambda^{-1} y$. 
  Denote by $B_r \subset \R^N$ the ball of radius $r>0$ centered at zero.
  It is proved in \cite[Theorem 3.1, Step 4 in Section 3.2]{DeLellis2017} that in the situation at hand there exists an oriented $2$--dimensional linear subspace $\pi \subset \R^N$ and $m \in \N$ such that for every sequence $(\lambda_n)_{n\in\N}$ converging to zero and $r>0$, the sequence $(T_{\lambda_n}|_{B_r})_{n\in\N}$ converges weakly to the Dirac delta $\delta_{m\pi}|_{B_r}$; $\delta_{m\pi}$ is called the \defined{tangent cone}.
  Note that this tangent cone is, in particular, a multiple of a $1$--branching in the sense of \autoref{Def_KBranching} for $f=0$.

  The crucial step in the proof is \cite[Theorem 1.8]{DeLellis2017b} which asserts the following.
  Suppose that for some $r>0$ the restriction $T|_{B_r}$ is \defined{approximated by a $k$--branching}: the precise definition is rather technical and is stated in \cite[Assumption 1.7]{DeLellis2017b}. 
  If this is the case, then 
  \begin{enumerate}
  \item
    either $T|_{B_\rho} = \ell G_{f,\phi}$ for some  $\rho >0$, $\ell \in \N$, and a $k$--branching $G_{f,\phi}$,
  \item
    or there are $k' > k$ and $\lambda >0$  such that $T_\lambda |_{B_r}$ is approximated by a $k'$--branching.
  \end{enumerate} 
  Now the theorem can be proved by induction.
  It follows from \cite[Theorem 3.1]{DeLellis2017} that for sufficiently small $\lambda > 0$ and $r>0$, $T_\lambda|_{B_r}$ is approximated by a $1$--branching, namely the tangent cone. 
  Thus, either $T_\lambda|_{B_\rho}$ is a multiple of a $1$--branching for some $\rho>0$, or after further rescaling is approximated by a $k$--branching with $k > 1$.
  If the latter is true, we can apply the theorem again and the process continues, so that $T$ is after rescaling approximated by a $k$--branching for larger values of $k$. 
  The process must, however, terminate at some point, because the sequence of rescaling converges to the tangent cone $\delta_{m\pi}$ from which it follows that $k \leq m$ \cite[third paragraph of Section 2.1]{DeLellis2017b}. 
  Thus, after finitely many steps, we obtain that $\lambda^*T|_{B_\rho} = \ell G_{f,\phi}$ for some $\lambda > 0$, $\rho>0$, $\ell \in \N$ and a $k$--branching $G_{f,\phi}$.
  We conclude that $T|_{B_{\lambda\rho}} = \ell G_{f,\tilde\phi}$ for $\tilde\phi \coloneq \lambda \cdot \phi$. 
\end{proof}


\section{$J$--holomorphic cycles and geometric convergence}

In this section we introduce the notions of a $J$--holomorphic cycle and geometric convergence.
We then compare these with the notions of an integral currents and weak convergence.
This comparison, combined with the results discussed in \autoref{Sec_RegularityTheory}, implies \autoref{Lem_JHolomorphicCycleCompactness}.

Throughout,
let $(M,J,g)$ be an almost Hermitian manifold.
Denote by
\begin{equation*}
  \sigma \coloneq g(J\cdot,\cdot)
\end{equation*}
the corresponding Hermitian form.
It follows from Wirtinger's inequality that $\sigma$ is a semicalibration on $(M,g)$ and that a $2$--dimensional submanifold is semicalibrated by $\sigma$ if and only if it is $J$--holomorphic \cites[Section 5.4.19]{Federer1969}[Section II.3 Example 1]{Harvey1982}. 

\begin{remark}
  If $\sigma$ is closed, that is $\sigma$ is a symplectic form and $J$ is compatible with $\sigma$, then $\sigma$ is a calibration and \autoref{Rem_Calibrations} recovers the well-known fact that $J$--holomorphic curves in symplectic manifolds minimize volume.
\end{remark}

\begin{definition}
~
  \label{Def_JHolomorphicCurve}
  \begin{enumerate}
  \item
    A \defined{$J$--holomorphic curve} is a subset of $M$ which is the image of a simple $J$--holomorphic map to $M$. 
    \label{Def_JHolomorphicCycle}
    A \defined{$J$--holomorphic cycle} $C$ is a formal linear combination
    \begin{equation*}
      C = \sum_{i=1}^I m_i C_i
    \end{equation*}
    of  $J$--holomorphic curves $C_1,\ldots,C_I$ with coefficients $m_1,\ldots,m_I \in \N$.
  \item
    The \defined{homology class} of a $J$--holomorphic curve is the homology class of the corresponding simple map and the homology class of a $J$--holomorphic cycle $C$ is
    \begin{equation*}
      [C] \coloneq \sum_{i=1}^I m_i [C_i].
    \end{equation*}
  \item 
    We say that $C$ is \defined{smooth} if  $C_1, \ldots, C_I$ are embedded and pairwise disjoint. 
  \item 
    For a $J$--holomorphic cycle $C$ denote by $\delta_C \co \Omega^2_c(M) \to \R$ the associated Dirac delta $2$--current defined in \autoref{Ex_DiracDelta}.
    The support $\supp(C)$ and mass $\bM(C)$ of $C$ are defined as the support and mass of $\delta_C$ respectively.
    Explicitly, 
    \begin{equation*}
      \supp(C) \coloneq \supp(\delta_C) = \bigcup_{i=1}^I C_i
      \qandq
      \bM(C) \coloneq  \bM(\delta_c)  = \sum_{i=1}^I m_i \area(C_i).
      \qedhere
    \end{equation*}
  \end{enumerate}
\end{definition}

\begin{definition}[{\citet[Definition 3.1]{Taubes1998}}]
  \label{Def_GeometricConvergence}
  Let $M$ be a manifold and let $(J_n,g_n)_{n \in \N}$ be a sequence of almost Hermitian structures converging to an almost Hermitian structure $(J,g)$ in the $C^\infty$ topology.
  For every $n \in \N$ let $C_n$ be a $J_n$--holomorphic cycle.
  We say that $(C_n)_{n \in \N}$ \defined{converges geometrically} to a $J$--holomorphic cycle $C$ if:
  \begin{enumerate}
  \item
    \label{Def_GeometricConvergence_Current}
    $(\delta_{C_n})_{n\in \N}$  converges weakly to $\delta_C$
    and
  \item
    \label{Def_GeometricConvergence_Support}
    $(\supp(C_n))_{n\in \N}$ converges to $\supp(C)$ in the Hausdorff distance;
    that is:
    \begin{equation}
      \label{Eq_GeometricConvergence_Support}
      \lim_{n\to \infty} d_H(\supp(C),\supp(C_n)) \to 0.
    \end{equation}
  \end{enumerate}
  Recall that the Hausdorff distance between two closed sets $X$ and $Y$ is defined by
  \begin{equation*}
    d_H(X,Y)
    \coloneq
    \max\set*{
      \sup_{x \in X} d(x,Y),
      \sup_{y \in Y} d(y,X)
    },
  \end{equation*}
  with $d$ denoting the distance with respect to the Riemannian metric $g$. 
\end{definition}

The following results compare $J$--holomorphic cycles and geometric convergence with closed $\sigma$--semicalibrated currents and weak convergence.

\begin{lemma}
  \label{Prop_FromCalibratedCurrentsToPseudoHolomorphicMaps}  
  If ~$T$ is a closed integer rectifiable current with compact support which is semicalibrated by $\sigma$,
  then there exists a $J$--holomorphic cycle $C$ such that 
  \begin{equation*}
    T = \delta_C.
  \end{equation*}
\end{lemma}

For symplectic $4$--manifolds this result was proved by \citet[Proposition 6.1]{Taubes1996}.
Taubes' argument and the work of \citet{Riviere2009} establish the result for symplectic manifolds, that is: when $\sigma$ is a calibration.

\begin{proof}
  Let $\mathring{\Sigma} = \reg(T)$ be the set of regular points of $T$, so that $\mathring{\Sigma}$ is an oriented $C^1$ submanifold. 
  Since $T$ is semicalibrated by $\sigma$, 
  the tangent spaces to $\mathring{\Sigma}$ are $J$--invariant.
  It follows from elliptic regularity that $\mathring{\Sigma}$ is a smooth submanifold and has a canonical structure of a Riemann surface.
  By  \autoref{Lem_SingularityStructure}, the singular locus $\sing(T)$ is discrete and so finite since $\supp(T)$ is compact. 
  Moreover,  every $x \in \sing(T)$ has a neighborhood $U$ such that
  \begin{equation*}
    \mathring{\Sigma} \cap U \iso D\setminus\set{0} \sqcup \cdots \sqcup D\setminus\set{0} \qandq \sing(T) \cap U = \set{x}.
  \end{equation*}
  Here $D = \set{ z \in \C : |z| < 1}$. 
  Thus, $\mathring{\Sigma}$ can be compactified to a Riemann surface $\Sigma$ by adding finitely many points.
  The compact Riemann surface $\Sigma$ comes with a continuous map $u \co \Sigma \to M$.
  Its restriction to $\mathring{\Sigma}$ is a smooth $J$--holomorphic embedding. 
  Since the image of $u$ is the compact set $\supp(T)$ and the energy of $u$ is $\bM(T) < \infty$, it follows from the removable singularity theorem \cite[Theorem 4.2.1]{McDuff2012}  that $u$ is, in fact, smooth and $J$--holomorphic on all of $\Sigma$.
  Moreover, the above discussion shows that
  \begin{equation*}
    T(\alpha) = \int_\Sigma m\cdot u^*\alpha \quad\text{for all } \alpha \in \Omega^2_c(M)
  \end{equation*}
  for some locally constant function $m \co \Sigma \to \N$.  
  Denoting by $C_1, \ldots, C_I$ the images of the connected components of $\Sigma$ and by $m_1, \ldots, m_I \in \N$ the corresponding values of $m$ yields 
  \begin{equation*}
    T = \delta_C \quad\text{with } C \coloneq \sum_{i=1}^I m_i C_i.
    \qedhere
  \end{equation*}
\end{proof}

\begin{lemma}\label{Prop_ConvergenceOfCurrentsImpliesGeometricConvergence}
  In the situation of \autoref{Def_GeometricConvergence},
  if condition \autoref{Def_GeometricConvergence_Current} holds and there exists a compact subset containing $\supp(C_n)$ for every $n \in \N$,
  then condition \autoref{Def_GeometricConvergence_Support} holds as well.
\end{lemma}

This result, which is an immediate consequence of the monotonicity formula for $J$--holomorphic maps, is contained in the proof of \cite[Proposition 3.3]{Taubes1998} and also a well-known fact in geometric measure theory, so we omit the proof. 

\begin{proof}[Proof of \autoref{Lem_JHolomorphicCycleCompactness}]
  Since $\sup_n \bM(C_n) < \infty$ and $\supp(C_n) \subset K$ for every $n \in \N$, by \autoref{Thm_FedererFleming}
  there exists a subsequence which   converges weakly to a closed integer rectifiable current semicalibrated by $\sigma$.
  By \autoref{Prop_FromCalibratedCurrentsToPseudoHolomorphicMaps},
  this current is of the form $\delta_C$ for a $J$--holomorphic cycle $C$.
  By \autoref{Prop_ConvergenceOfCurrentsImpliesGeometricConvergence},
  the sequence of pseudo-holomorphic cycles $(C_n)$ geometrically converges to $C$.
\end{proof}


\section{Proof of \autoref{Thm_KRigidityImpliesFiniteness}}
\label{Sec_KRigidityImpliesFiniteness}

Suppose that $J$ is $k$--rigid and that
$A \in H_2(M)$ satisfies $\inner{c_1(M,J)}{A} = 0$ and its divisibility is at most $k$.
If the conclusion of the theorem fails,
then there are are infinitely many \emph{distinct} $J$--holomorphic curves $C_n \subset M$ representing $A$ and of energy at most $\Lambda$.
By \autoref{Lem_JHolomorphicCycleCompactness},
after passing to a subsequence, the sequence $(C_n)$ converges geometrically to a $J$--holomorphic cycle
\begin{equation*}
  C_\infty = \sum_{i=1}^I m_i C_\infty^i.
\end{equation*}

\begin{prop}
  \label{Prop_CinftyRegularity}
  $C_\infty$ is connected, smooth, and its multiplicity is at most the divisibility of $A$.
\end{prop}

\begin{proof}
  By \autoref{Def_GeometricConvergence}~\autoref{Def_GeometricConvergence_Current},
  $[C_\infty] = [A]$.
  Let $u_i \co \Sigma_i \to M$ be a simple $J$--holomorphic map whose image is $C_\infty^i$.
  The index formula \autoref{Eq_Index} yields
  \begin{equation*}
    \sum_{i=1}^I m_i \ind(u_i)
    = \sum_{i=1}^I 2 m_i \inner{c_1(M,J)}{[C_\infty^i]}
    = 2 \inner{c_1(M,J)}{[C_\infty]}
    = 0.
  \end{equation*}
  Since $J$ is $k$-rigid, by \autoref{Def_KRigidAlmostComplexStructure}~\autoref{Def_KRigidAlmostComplexStructure_NonNegativeIndex},
  there are no $J$--holomorphic curves of negative index.
  Thus, we have $\ind(u_i) \geq 0$ for every $i \in \set{1,\ldots,I}$ and the above computation shows that
  \begin{equation*}
    \ind(u_1) = \cdots = \ind(u_I) = 0.
  \end{equation*}
  Therefore,
  by \autoref{Def_KRigidAlmostComplexStructure}~\autoref{Def_KRigidAlmostComplexStructure_Embedding},
  the $J$--holomorphic curves $C_\infty^1,\ldots,C_\infty^I$ are embedded and pairwise disjoint.
  This proves that $C_\infty$ is smooth.

  To see that $C_\infty$ is connected, observe that if $C_\infty$ were disconnected, then \autoref{Def_GeometricConvergence}~\autoref{Def_GeometricConvergence_Support} would imply that $C_n$ is disconnected for $n \gg 1$.
  However, $C_n$ is a $J$--holomorphic curve and thus connected by definition.

  Since $A = m_1[C_\infty^1]$, it follows that $m_1$ is at most the divisibility of $A$.
\end{proof}

In the following,
we rescale the sequence $(C_n)$ and extract a further limit $\tilde C_\infty$.
The properties of $\tilde C_\infty$ will give a contradiction to $J$ being $k$--rigid.

Henceforth,
we denote by $C_\infty^1$ the $J$--holomorphic curve underlying the $J$--holomorphic cycle $C_\infty$.
Since the curves $C_n$ are all distinct, we can assume that they are all distinct from $C_\infty^1$.
We can also assume that every $C_n$ is contained in a sufficiently small tubular neighborhood of $C_\infty^1$.
By slight abuse of notation,
we regard $C_n$ as an $\exp^*\!J$--holomorphic curve in the normal bundle $NC_\infty^1$
and $C_\infty^1$ as the zero section in $NC_\infty^1$.

For every $\lambda > 0$ let $\sigma_\lambda$ be as in \autoref{Prop_RescalingComplexStructure}.
Choose $(\lambda_n)$ such that such that the sets
\begin{equation*}
  \tilde C_n \coloneq \sigma_{\lambda_n}^{-1}(C_n)
\end{equation*}
satisfy
\begin{equation}
  \label{Eq_DTildeCC=1}
  d_H(\tilde C_n,C_\infty^1)
  = 1/2.
\end{equation}
Set
\begin{equation*}
  J_n \coloneq \sigma_{\lambda_n}^*\exp^*J.
\end{equation*}
By construction,
the $\tilde C_n$ are $J_n$--holomorphic.
By \autoref{Prop_RescalingComplexStructure}, the sequence $(J_n)$ converges to the almost complex structure $J_u$ associated with the $J$--holomorphic map $u \co C_\infty^1 \hookrightarrow M$.
The sequence $(\tilde C_n)$ is contained in the compact disc bundle $\bar B_{1/2}(NC_\infty^1) \subset NC_\infty^1$.
By \autoref{Prop_PropertiesOfJD}~\autoref{Prop_PropertiesOfJD_Tamed},
$J_u$ is tamed by a symplectic form $\omega$ on $B_1(NC_\infty^1)$.
Consequently, for $n \gg 1$ the almost complex structure $J_n$ is tamed by $\omega$ as well.
Define a Riemannian metric  $g$ on $B_1(NC_\infty^1)$ by
\begin{equation*}
  g \coloneq \frac12(-\omega(J_u\cdot,\cdot)+\omega(\cdot,J_u\cdot)).
\end{equation*}
The analogously defined metrics $g_n$ are Hermitian with respect to $J_n$ and converge to $g$.
By the energy identity \cite[Lemma 2.2.1]{McDuff2012},
\begin{equation*}		
  \lim_{n\to \infty} \bM(\tilde C_n)
  =
  \lim_{n\to \infty} \delta_{\tilde C_n}(\omega)
  =
  \delta_{C_\infty^1}(\omega) < \infty.
\end{equation*}
Therefore, the mass of  $\tilde C_n$ with respect to $g_n$ (and thus also $g$) can be bounded independent of $n$.

By \autoref{Lem_JHolomorphicCycleCompactness},
a subsequence of $(\tilde C_n)_{n\in\N}$ geometrically converges to a $J$--holomorphic cycle
\begin{equation*}
  \tilde C_\infty = \sum_{i=1}^I \tilde m_i \tilde C_\infty^i.
\end{equation*}
Let $d_i \in \N$ be such that $[\tilde C^i_\infty] = d_i [C^1_\infty]$.
Condition \autoref{Eq_DTildeCC=1} guarantees that $\supp(\tilde C_\infty) \neq C_\infty^1$.
Therefore, without loss of generality, $\tilde C_\infty^1$ is not contained in the zero section.
Since
\begin{equation*}
   m_1[C_\infty^1] = A = [\tilde C_\infty]
  = \sum_{i=1}^I \tilde m_i d_i [C_\infty^1],
\end{equation*}
we have $d_1 \leq \tilde m_1 d_1 \leq m_1 \leq k$. 
\autoref{Prop_CurvesInNormalBundle} applies and the map $\varphi \co \tilde C^1_\infty \to C^1_\infty$ defined there has degree $d_1$. 
This contradicts $J$ being $k$--rigid.
\qed


 
\printreferences

\end{document}
